\documentclass{amsart}
\usepackage{amsmath}
\usepackage{amssymb}
\usepackage{graphicx}
\input xy
\xyoption{all}
\newtheorem{theorem}{Theorem}[section]

\newtheorem{lemma}[theorem]{Lemma}
\newtheorem{remark}[theorem]{Remark}

\theoremstyle{definition}

\newcommand{\Cal}[1]{{\mathcal #1}}
\newcommand{\sign}{\mbox{\rm sign}}

\newcommand{\End}{\mbox{\rm End}}

\newcommand{\C}{\mathbb{C}}

\newcommand{\R}{\mathbb{R}}

\newcommand{\GL}{{\rm GL}}
\newcommand{\M}{{\rm M}}

\newcommand{\Char}{{\rm char}}

\begin{document}
    \title[Algebraic normal subgroups]{Algebraic commutators with respect to subnormal subgroups in division rings}
    \author[M. H. Bien]{Mai Hoang Bien$^{1,2}$}\author[B. X. Hai]{Bui Xuan Hai$^{1,2}$}\author[V. M. Trang]{Vu Mai Trang $^{1,2}$} 
    \thanks{The first and the second authors are  funded by  Vietnam National University HoChiMinh City (VNUHCM) under grant number B2020-18-02}
    \email{mhbien@hcmus.edu.vn;bxhai@hcmus.edu.vn;trangvm8234@gmail.com}
\address{[1] Faculty of Mathematics and Computer Science, University of Science, Ho Chi Minh City, Vietnam.}
\address{[2] Vietnam National University, Ho Chi Minh City, Vietnam.}


\maketitle
\begin{abstract}{Let $D$ be a division ring and $K$ a subfield of $D$ which is not necessarily contained in the center $F$ of $D$. In this paper, we study the structure of $D$ under the condition of left algebraicity of certain subsets of $D$ over $K$. Among results, it is proved that if $D^*$ contains a noncentral normal subgroup which is left algebraic over $K$ of bounded degree $d$, then $[D:F]\le d^2$. In case $K=F$, the obtained results show that if either all additive commutators or all multiplicative commutators with respect to a noncentral subnormal subgroup of $D^*$ are algebraic of bounded degree $d$ over $F$, then $[D:F]\le d^2$.}		

\end{abstract}

\section{Introduction}

Let $D$ be a division ring with center $F$. If $a\in D$ is an algebraic element over $F$, then we say that  $d=[F(a):F]$ is the \textit{algebraic degree} of $a$ and $a$ is \textit{algebraic of degree $d$}. A subset $S$ is {\it algebraic of bounded degree} if there exists a positive integer $d$ such that every element of $S$ is algebraic of degree $\le d$.  In 1945, Jacobson proved that if $D$ is algebraic  of bounded degree over $F$, then $D$ is centrally finite; that is, $[D:F]<\infty$ (see \cite[Theorem 7]{Pa_Ja_45}). This fact has been considered as one of  the classical results in the theory of division rings. Further, it was an inspiration for several related topics: 
for the study  concerning left (right) algebraicity in a division ring  over some division subring, see e.g. \cite{Pa_BeRo_14,Pa_BeDrSh_13,Pa_HaTrBi_19, Bo_Ja_64}; for the study of algebraic valued functions, see e.g. \cite{Pa_HePrSc_75}; for the influence of algebraic commutators of bounded degree on the structure of a division ring, see e.g.  \cite{Pa_Bi_16,Pa_DeBiHa_19,Pa_Di_86,Pa_Fa_60,Pa_GoMa_86,Pa_GoSh_08,Pa_HaKhBi_2020,Pa_He_82, Pa_He_80, Pa_MaMa_85} and references therein. After Jacobson's theorem we have mentioned above, it is natural to evaluate the dimension $[D:F]$ for a division ring $D$ with these properties. It follows from \cite[Theorem 6]{Pa_ChFoLe_04} and \cite[Theorem~ 1.3]{Pa_BeDrSh_13} that if $D$ is algebraic of bounded degree $d$ over $F$, then $[D:F]\le d^2$.  
 
Further, instead of the algebraicity of $D$,  several authors have been considering only the algebraicity of the set of commutators (either multiplicative or additive) or the symmetric set of $D$ (see e.g. \cite{Pa_AkArMe_98, Pa_ChFoLe_04, Pa_He_82, Pa_He_80, Pa_HePrSc_75, Pa_Ma_00, Pa_MaAkMeHa_95, susan_algebraic, Pa_Sl_64}). 

The first purpose of this paper is to evaluate $[D:F]$ in case $D^*$ contains a non-central subnormal subgroup $N$ with one of the following properties:

1) $N$ is normal and  algebraic of bounded degree $d$ over $F$;

2) all multiplicative commutators of the form $aba^{-1}b^{-1}$, where $a\in N$ and $b\in D^*$, are algebraic of bounded degree $d$ over $F$;

3) all additive commutators of the form $ac-ca$, where $a\in N$ and $c\in D^*$, are algebraic of bounded degree $d$ over $F$.

It is surprising that in any of these cases, the result we get always confirms that $[D:F]\le d^2$ (see theorems \ref{c4.3}, \ref{th:4.2}, and \ref{th:4.3} in Section 4). These results strongly generalize 
\cite[Theorem 7]{Pa_Ja_45}, \cite[Theorem 6]{Pa_ChFoLe_04}, and \cite[Theorem 2]{Pa_AkArMe_98}. Moreover, Theorem 3.3 assures that this evaluation is best possible. 

Our second purpose is to study the left algebraicity of bounded degree of a non-central subnormal subgroup $N$ of $D^*$ over a subfield $K$ of $D$ which is not necessarily contained in the center $F$ of $D$. For this purpose, in Section 2, we recall the concept of left (right) algebraicity in noncommutative rings. In this section, we recall also some useful properties of generalized Laurent identities in  matrix rings over division rings we need for our further study. Section~ 3  is devoted to the proof of the existence of maximal subfields  in a centrally finite division ring $D$ which are simple extensions over the center $F$ of $D$. Moreover, they are generated by one commutator which may be either multiplicative or additive. The results we get in this section are mainly applicable to deduce the theorems of Section 4 as we have mentioned above. Finally, in Section 5, we study the left algebraicity in a division ring $D$ over its arbitrary subfield $K$. The main result we get in Theorem \ref{left algebraic} is really general. In fact, we have successfully evaluated the dimension $[D:F]$ of $D$ in case $D^*$ contains a non-central normal subgroup $N$ which is left algebraic of bounded degree $d$ over a subfield $K$ of $D$, where $K$ is not necessarily contained in the center $F$ of $D$. More exactly, this evaluation is $[D:F]\le d^2$. Of course, Theorem \ref{c4.3} is a particular case of this result which generalizes \cite[Theorem 1.3]{Pa_BeDrSh_13} as well. 

\section{Preliminaries} 

Let $K$ be a field and $R$ a ring containing $K$ as a subring. We say that an  element $a\in R$ is \textit{left algebraic} over $K$ if there exists a non-zero polynomial $f(t)=a_nt^n+a_{n-1}t^{n-1}+\dots+a_1t+a_0\in K[t]$ such that $a$ is a right root of $f(t)$; that is,  $f(a)=a_na^n+a_{n-1}a^{n-1}+\dots+a_1a+a_0=0$. Additionally, if $f(t)$ is a monic polynomial of  smallest degree $n$ such that $f(a)=~0$, then it is easy to see that $f(t)$ is unique. In this case, we say that $n$ is  the \textit{left algebraic degree}, and $f(t)$ is  the \textit{left minimal polynomial} of $a$. For a positive integer $d$, we say that $a$ is \textit{left algebraic of bounded degree} $d$ over $K$ if $a$ is left algebraic over $K$ of degree not bigger than $d$. The notions of the \textit{right algebraicity}, the \textit{right algebraic degree}, the \textit{right minimal polynomial} and the \textit{right algebraicity of bounded degree} are defined similarly. If $K$ is contained in the center $Z(R)$ of $R$, then every right notion coincides with its corresponding left notion, so in this case, the prefix ``right" or ``left" can be omitted. 

Note that the left (right) minimal polynomial of a left (right) algebraic element over $K$ is unique but it is not necessarily irreducible in general, even if $K$ is assumed to be contained in the center $Z(R)$. For example, let $K$ be a field,  $R=\mathrm{M}_2(K)$ be the matrix ring of degree $2$ over $K$, and  $a=\begin{bmatrix}
1 & 1  \\
0 & 1 \\
\end{bmatrix}$. Then, the minimal polynomial of $a$ is $(t-1)^2$. For an another example, let $\mathbb{H}=\R \oplus \R i\oplus \R j\oplus \R k $ be the division ring of real quaternions. Clearly, $t^2+1$ is the left minimal polynomial of $j$ over $\C$ which is not irreducible over $\C$. However, if $R=D$ is a division ring and $K$ is a subfield of the center $Z(D)$ of $D$, then the minimal polynomial of an algebraic element over $K$ is irreducible (and unique). 

Now, let $F$ be a field, $R$ a ring whose center contains $F$ as a subfield and $X=\left \{ x_{1},x_{2},\dots ,x_{m} \right \}$ a set of $m$ noncommuting indeterminates. Denote by $F\left \langle X \right \rangle$ the group algebra of the free group $\langle X\rangle$ over $F$ and $R_F\left \langle X \right \rangle=R*_FF\left \langle X \right \rangle$ the free product of $R$ and $F\left \langle X \right \rangle$ over $F$. Then,  each element of $R_F\left \langle X \right \rangle$ can be written in the form $f(X)=\sum_{i=1}^{n}f_i(X)$ with some positive integer $n$ and  $$f_i(X)=a_{i1}x_{i_1}^{m_{i_1}}a_{i2}x_{i_2}^{m_{i_2}} \dots a_{it}x_{i_t}^{m_{i_t}}a_{i(t+1)},$$ where $x_{i_j}\in X$, $m_{i_j}\in \{-1,1\}, a_{ij}\in R$. Each element of $R_F\left \langle X \right \rangle$ is called a \textit{generalized Laurent polynomial} in $X$ of $R$ over $F$. 
Let $f\left(X\right)\in R_F\langle X\rangle$. For $m$  elements $c_1,c_2,\dots,c_m$ of the unit group $R^*$ of $R$, $f(c_1, c_2, \dots, c_m)$ is understood to be the value of $f$ when replacing $x_{i}$ by $c_{i}$. Assume that $f$ is non-zero. We say that $f=0$ is a \textit{generalized Laurent identity} of $R$ if $f(c_1,c_2,\dots,c_m)=0$ for all $c_1,c_2,\dots,c_m\in R^*$. Let $\Cal I(R)$ denote the set of all generalized Laurent identities of $R$. The following lemma is a special case of the Amitsur-Bergman Theorem \cite[Theorem 8.2.15]{Bo_Rowen_80} for division rings. 

\begin{lemma}\label{l2.3}
	Let $D$ be a division ring with center $F$. If $[D:F]=n^2$, $K$ is a field containing $F$ and $D\subseteq \M_n(K)$, then $\Cal I(D)\subseteq \Cal I(\M_n(K))$.
\end{lemma}
\begin{proof} This lemma is just a corollary of \cite[Lemma 8]{Pa_ChFoLe_04} (or see \cite[Theorem 6.4]{Pa_Ro_84}).
\end{proof}

For a natural number $d$, we put
\begin{align}\label{e1}
g_d(x,y_1,y_2,\dots,y_d)=\sum\limits_{\sigma\in S_{d+1}}\sign(\sigma)x^{\sigma(0)}y_1 x^{\sigma(1)}y_2\dots x^{\sigma(d-1)}y_dx^{\sigma(d)},
\end{align}  where $S_{d+1}$ is the symmetric group on the set $\{0,1,\dots,d\}$ and $\sign(\sigma)$ is the sign of $\sigma\in S_{d+1}$. It is clear that $g(x,y_1,\cdots,y_d)$ is a generalized Laurent polynomial. 

\begin{lemma}\label{l2.2} {\rm \cite[Corollary 2.3.8]{Bo_BeMaMi_96}}
	Let $D$ be a division ring with center $F$ and $\M_n(D)$ the matrix ring of degree $n$ over $D$. For any $a\in \M_n(D)$, the following conditions are equivalent:
	\begin{enumerate}
		\item $a$ is algebraic of degree $\le d$ over $F$.
		\item $g_{d}(a,b_1,b_2,\dots,b_d)=0$ for all $b_1,b_2,\dots,b_d\in \M_n(D)$.
	\end{enumerate}
\end{lemma}

 Let $R$ be a ring whose center contains a field $F$. As we have mentioned above, if $a\in R$ is algebraic over $F$, then the minimal polynomial $p_a(t)$ of $a$ over $F$ is not necessarily irreducible. 
 If $p_a(t)$ is irreducible, then it is obvious that  the subring $F[a]$ of $R$ generated by $a$ over $F$ is a field. For the convenience of use, we restate this fact as a lemma.
\begin{lemma}\label{l3.1} Let $F$ be a field and $R$ an algebra whose center  contains $F$ as a subfield. Assume that $a\in R$ is algebraic over $F$. If $p_a(t)$ is irreducible over $F$, then the subring $F[a]$ of $R$ generated by $a$ over $F$ is a field.
\end{lemma}

Let $K$ be a field and $R$ a ring containing $K$ as a subring. Assume that $R$ is a right vector space over $K$ of dimension $n$. Then, $R$ can be considered as a subring of $\End(R_K)$ via the injective ring homomorphism $\varphi: R\longrightarrow \End(R_K), \varphi(a)=\varphi_a$,  where $\varphi_a(x)=ax$ for every $x\in R$. Moreover, with respect to some basis of $R$ over $K$,  we have an isomorphism $\End (R_K)\stackrel{\psi}{\cong} \M_n(K)$. Recall that  
$$\tau=\psi\varphi: R\longrightarrow \M_n(K)$$
is called a \textit{right regular representation} of $R$. Due to  this representation, $R$ can be viewed as a subring of $\M_n(K)$. For an element $\alpha\in R$, by identifying $\alpha$ with its image $A_{\alpha}=\tau(\alpha)$, sometimes it is convenient to view $\alpha$ as an element of the matrix ring $\M_n(K)$.

\begin{remark}\label{r1}{\rm 	Let $D$ be a division ring with center $F$ and $K$ a maximal subfield of $D$. 

1. It is clear that $F\subseteq K$ and the inclusion is strict if $D$ is noncommutative. If $[D:F]=n^2$, then $\dim(D_K)=n$ (see  \cite[Theorem (15.8)]{Bo_La_91}). By the previous arguments, $D$ can be viewed as an $F$-subalgebra of $\M_n(K)$ via some right regular representation  $\tau$ of $D$ with respect to some basis. Note that the center of $D$ is $F$  but the center of $\M_n(K)$ is $K$. 
		
2. Let $\alpha\in D$. Assume that $p_{F, \alpha}(t)\in F[t]$ is the minimal polynomial of $\alpha$ over $F$. Viewing $\alpha$ as an element in $\M_n(K)$ via $\tau$, by \cite[Lemma 16.3 (1) and Lemma~ 16.4]{Bo_BeOg_13}, the minimal polynomial $p_{K,\alpha}(t)$ of $\alpha$ over $K$ coincides with $p_{F, \alpha}(t)$. For this reason, from now on, for $\alpha \in D$, we use the simple notation  $p_{\alpha}(t)$ to denote the minimal polynomial of $\alpha$ either over $F$ or over $K$.}
\end{remark}

For further use, we need some notations. Let $n_1,n_2,\dots,n_t$ be $t$ positive integers and $F$ a field. Assume that 
$$A_1\in \mathrm M_{n_1}(F), A_2\in \mathrm M_{n_2}(F),\dots, A_t\in \mathrm M_{n_t}(F).$$ 
We denote by $A_1\oplus A_2\oplus \dots \oplus A_t$ the matrix $$\begin{bmatrix}
A_1 & 0 & \dots & 0 \\
0 & A_2 & \dots & 0\\
\vdots & \vdots & \ddots & \vdots \\
0 & 0 &\dots & A_t\\
\end{bmatrix} \in \mathrm{M}_{n_1+n_2+\dots+n_t}(F),$$ and by $e_{ij}$ the matrix in which the $(i,j)$-entry is $1$ and all other entries are $0$.

\begin{remark}\label{r2}
	{\rm Given a monic polynomial $f(t)=a_0+a_1t+\dots+a_{n-1}t^{n-1}+t^{n}\in F[t]$, the \textit{Frobenius companion matrix} $C_f$ of  $f(t)$ is $$C_f=\sum_{i=1}^{n-1} e_{i+1,i} -\sum_{i=1}^n {a_{i-1} e_{i,n}}=\begin{bmatrix}
		0 & 0 & 0&  \dots & 0& -a_0 \\
		1 & 0 &  0&\dots & 0& -a_1\\
		0 & 1 &  0&\dots & 0&-a_2\\
		\vdots & \vdots & \ddots & & \vdots \\
		0 & 0 &  0&\dots & 0&-a_{n-2}\\
		0 & 0 &0 & \dots & 1& -a_{n-1}\\
		\end{bmatrix} \in \mathrm{M}_{n}(F).$$ It is well known that the characteristic polynomial and the minimal polynomial of $C_f$  coincide and they are both equal to $f(t)$. Moreover, if $A$ is a matrix in $\mathrm{M}_n(F)$ whose  characteristic polynomial and minimal polynomial  coincide and equal to  $f(t)$, then $A$ is similar to $C_f$ (see \cite[Theorem 3.3.15]{Bo_HoJo_85}). Recall that two matrices $A, B\in \M_n(F)$ are similar (written $A\sim B$) if there exists a matrix $P\in\GL_n(F)$ such that $A=P^{-1}BP$. }
\end{remark}

\begin{lemma}\label{l3.3} Let $L/F$ be a separable field extension of finite degree $n$. Assume that $\alpha$ is an element from $L$,   $p_\alpha(t)=a_0+a_1t+\dots+a_{d-1}t^{d-1}+t^{d}\in F[t]$ is the  minimal polynomial of $\alpha$, and $$C_p=\sum_{i=1}^{d-1} e_{i+1,i} -\sum_{i=1}^d {a_{i-1} e_{i,d}} \in \mathrm{M}_{d}(F)$$ is the Frobenius companion matrix of $p_\alpha(t)$. Then, $n=kd$ and there exists a basis $\Cal B$ of $L$ over $F$ such that via right regular representation of $L$ with respect to $\Cal B$, the corresponding matrix $A_\alpha\in \M_n(F)$ of $\alpha$ is $A_\alpha=C_p\oplus C_p\oplus \dots \oplus C_p$ ($k$ times).
\end{lemma}
\begin{proof} It is trivial that $n=kd$, where $k=[L:F(\alpha)]$. Since $L/F$ is separable, so is $L/F(\alpha)$. Let $\beta\in L$ such that $L=F(\alpha) (\beta)$. Then, consider the basis $$\Cal B=\{1, \alpha, \dots, \alpha^{d-1}, \beta, \alpha\beta,\dots, \alpha^{d-1}\beta, \dots,\beta^{k-1}, \alpha\beta^{k-1},\dots, \alpha^{d-1}\beta^{k-1} \}$$ of $L$ over $F$.  By checking one by one elements in the basis $\Cal B$ for the right regular representation of $\alpha$ with respect to $\Cal B$, one has  $A_\alpha=C_p\oplus C_p\oplus \dots \oplus C_p$ ($k$ times). 
\end{proof}
The following result follows immediately from \cite[corollaries 1 and 2]{Pa_ChFoLe_04}.
\begin{lemma}\label{l3.4} Let $n$ be a positive integer greater than $2$ and $n=n_1+n_2+\dots+n_d$, where $n_i>1$. For every $1\le i\le d$, assume that $C_i$ is the Frobenious companion matrix of some polynomial over a field $F$. If $C=C_1\oplus C_2 \oplus \dots \oplus C_d\in \mathrm {M}_{n}(F)$, then there exist $A\in \GL_n(F)$ and $B\in \mathrm{M}_n(K)$ such that  $CAC^{-1}A^{-1}$ and  $BC-CB$ are algebraic over $F$ of degree $n$.
\end{lemma}

\begin{remark} \label{r3} {\rm Let $F$ be a field. Assume that $L$ is a subfield of $\mathrm \M_n(F)$ containing $F$ such that $\dim_FL=n$. Via the right regular representation $\tau$ with respect to some basis $\Cal B$ of $L$ over $F$, we can view $L$ as a subring of $\M_n(F)$.  According to \cite[Lemma 5]{Pa_ChFoLe_04}, if $a$ is an element from $L$, then $a$ and $A_a:=\tau(a)$ are similar; that is, $a\sim A_a$.}
\end{remark}

\section{Maximal subfields generated by commutators}
Let $D$ be a centrally finite division ring with center $F$.  By \cite[Theorem (15.12)]{Bo_La_91} and \cite[Corollary 5.7]{Bo_morandi_96}, there exists a separable over $F$ element $a\in D$ such that $F(a)$ is a maximal subfield of $D$. Let $\Cal M$ denote the set of all elements $a$ in $D$ such that $F(a)$ is a maximal  subfield of $D$. It is natural to ask how big is the set $\Cal M$?  There are many reasons why this question is interesting. For example, Albert-Brauer's theorem  says that if $a\in \Cal M$, then there exists $b\in D$ such that the set $\{a^ib a^j\mid 1\le i,j\le n\}$ is a basis of the vector space $D$ over $F$ provided $[D:F]=n^2$ \cite[Theorem (15.16)]{Bo_La_91}. As we have mentioned above, the first conclusion is that $\Cal M$ contains some separable over $F$ element of $D$. In \cite{Pa_MaAkMeHa_95}, it was shown that the intersection of $\Cal M$ and the $r$-th derived subgroup $D^{(r)}$ is non-empty. Although they are not redirect results, using the ideas in \cite[theorems 4 and 6]{Pa_ChFoLe_04}, we can show that there exist $a,b,c,d\in D^*$ such that $ab-ba$ and $cdc^{-1}d^{-1}$ belong to $\Cal M$, which answered affirmatively questions of M. Mahdavi-Hezavehi in \cite[problems 28,  29]{Pa_Ma_00}. With the same ideas, we also prove that if $\Char(D)=0$, then, for any $a\in D^*$, there exist $b,c\in D^*$ such that $ab-ba$ and $aca^{-1}c^{-1}$ are in $\Cal M$. Some recent results on this topic were presented in \cite{Pa_AaBi_19}. The aim of this section is to show that if $N$ is a non-central subnormal subgroup of $D^*$, then there exist $a\in N$ and $b,c\in D^*$ such that $aba^{-1}b^{-1}$ and $ac-ca$ are in $\Cal M$.

The following fact is contained in the proof of \cite[Theorem 9]{Pa_He_78}. For the convenience, we restate it as a lemma. 
\begin{lemma} \label{t2.1} Let $D$ be a division ring with center $F$ and $N$ a subnormal subgroup of $D^*$. If $N$ is purely inseparable over $F$, then $N$ is central; that is, $N\subseteq F$.
\end{lemma}
\begin{proof} See the end of the proof of \cite[Theorem 9]{Pa_He_78}.
\end{proof} 
The following result can be considered as  an extension of the Noether-Jacobson Theorem (see \cite[Theorem (15.11)]{Bo_La_91}) for subnormal subgroups in division rings.
\begin{lemma}\label{lem:3.2}
	Let $D$ be a division ring with center $F$ and $N$ a subnormal subgroup of $D^*$. If $N$ is non-central and algebraic over $F$, then $N$ contains a non-central separable element over $F$. 
\end{lemma}
\begin{proof} Clearly, it suffices to consider the case of prime characteristic; that is, $\Char(D)=p$ is a prime. Assume that  $N$ is a non-central subnormal subgroup of $D^*$ which is algebraic over $F$. According to Lemma~\ref{t2.1}, $N$ contains an element $a$ which is not purely inseparable over $F$. By \cite[Proposition 4.6]{Bo_morandi_96}, there exists a positive integer $n$ such that $a^{p^n}$ is separable over $F$. Since $a$ is not purely inseparable over $F$, we deduce that $a^{p^n}\notin F.$
\end{proof}

The following theorem is the main result of this section.
\begin{theorem}\label{t3.3} Let $D$ be a centrally finite division ring with center $F$. Assume that $N$ is a non-central subnormal subgroup of $D^*$. Then, there exist $a\in N$ and $b,c\in D^*$  such that $F(aba^{-1}b^{-1})$ and $F(ac-ca)$ are maximal subfields of $D$. 
\end{theorem}
\begin{proof} Let $[D:F]=n^2$. By Lemma~\ref{lem:3.2}, we can find a non-central element $a\in N$ such that $a$ is  separable over $F$. In view of  \cite[Theorem (15.8)]{Bo_La_91}, it suffices to show that there exist $b,c\in D^*$ such that $aba^{-1}b^{-1}$ and $ac-ca$ are algebraic  elements of degree $n$ over $F$. 
	\bigskip
	
	{\it Case 1. $n=2$.} Let $\alpha\in D^*$ such that $a\alpha\ne \alpha a$. If $\beta a\beta^{-1} a^{-1}\in F$ for every $\beta\in D^*$, then $$(\alpha+1)\left( (\alpha+1)^{-1} a (\alpha+1)a^{-1}-(\alpha^{-1}a\alpha a^{-1})  \right) $$$$= a (\alpha+1)a^{-1}-(\alpha+1)\alpha^{-1}a\alpha a^{-1}$$ $$=a\alpha a^{-1}+1-a\alpha a^{-1}-\alpha^{-1}a\alpha a^{-1}$$ 
	$$  =1-\alpha^{-1}a\alpha a^{-1}\in F^*.$$ 
This implies that $\alpha\in F^*$, a contradiction.
Hence, there exists $b\in D^*$ such that $b ab^{-1} a^{-1}$ is algebraic of degree $2$ over $F$. Similarly, if $a\beta-\beta a\in F$ for every $\beta\in D$, then $a (a\alpha)-(a\alpha )a = a (a\alpha-\alpha a)$. Observe that $a\alpha-\alpha a\ne 0$ and $a (a\alpha)-(a\alpha )a\in F$, so $$a=(a (a\alpha)-(a\alpha )a)(a\alpha-\alpha a)^{-1}\in F,$$ which is a contradiction. Hence, there exists $c\in D^*$ such that $ac-ca\not\in F$. But this implies that $ac-c a$ is algebraic of degree $2$ over $F$. 
	\bigskip
	
	{\it Case 2. $n>2$.} Since $a$ is separable over $F$, so is $F(a)$. According to \cite[Theorem (15.12)]{Bo_La_91}, there exists a maximal subfield $K$ of $D$ which contains $F(a)$ and is separable over $F$. By Remark~\ref{r1}, $D$ is an $F$-subalgebra of $\M_n(K)$. We first claim that there exists $v\in \GL_n(K)$ and $u\in \mathrm{M}_n(K)$ such that $ava^{-1}v^{-1}$ and $au-ua$ are algebraic over $F$ of degree $n$. Assume that the minimal polynomial of $a$ over $F$ is $$p_a(t)=a_0+a_1t+\dots + a_{d-1}t^{d-1}+t^d.$$  
Because the extension $K/F$ is finite and separable, there exists $b\in K$ such that $K=F(b)$ (see e.g. \cite[Corollary 5.7]{Bo_morandi_96}). By \cite[Theorem (15.8)]{Bo_La_91}, $\dim_FK=n$, so the minimal polynomial of $b$ over $F$ is of degree $n$, namely, $$p_b(t)=b_0+b_1t+\dots+b_{n-1}t^{n-1}+t^{n}.$$ In view of Remark~\ref{r1} (2), $p_b(t)$ is also the minimal polynomial of $b$ over  $K$.  Consider the Frobenius companion matrix $$c=C_{p_{b}}=\sum_{i=1}^{n-1} e_{i+1,i} -\sum_{i=1}^n {b_{i-1} e_{i,n}}=\begin{bmatrix}
	0 & 0 & 0&  \dots & 0& -b_0 \\
	1 & 0 &  0&\dots & 0& -b_1\\
	0 & 1 &  0&\dots & 0&-b_2\\
	\vdots & \vdots & \ddots & & \vdots \\
	0 & 0 &  0&\dots & 0&-b_{n-2}\\
	0 & 0 &0 & \dots & 1& -b_{n-1}\\
	\end{bmatrix} $$ of $p_{b}(t)$. By Remark~\ref{r2}, there exists $P\in \GL_n(K)$ such that $c=P^{-1}b P$. Consider the subring $\mathrm M_{n}(F)$ of $\mathrm \M_n(K)$. Note that $c\in \mathrm M_{n}(F)$. Let $L=F[c]$ be the subring of $\M_n(F)$ generated by $c$ over $F$. Since the minimal polynomial of $c$ is $p_c(t)=p_b(t)$ which is irreducible of degree $n$ over $F$, by Lemma~\ref{l3.1}, $L$ is a subfield of $\mathrm M_{n}(F)$ and $[F(c):F]=n$. Observe that $a\in F[a]\subseteq F[b]$, so $$P^{-1}aP\in P^{-1}F[a]P\subseteq P^{-1}F[b]P=F[P^{-1} b P]=F[c]=L.$$ For convenience, put $\delta=P^{-1}aP$. The minimal polynomials of $\delta$ and $a$ coincide, so $p_{\delta}(t)=p_{a}(t)=a_0+a_1t+\dots+a_{d-1}t^{d-1}+t^d$. Hence, if $C_{p_{\delta}}$ is the Frobenius companion matrix of $p_{\delta}$, then, by Lemma~\ref{l3.3}, there exists a basis $\Cal B$ such that the corresponding matrix $A_\delta$ of $\delta$ is similar to $$E=C_{p_{A}}\oplus C_{p_{A}}\oplus \dots \oplus C_{p_{A}}$$ ($n/d$ times) in $\mathrm \M_n(F)$. 
	Combining this fact with Remark~\ref{r3}, one has $\delta\sim A_\delta\sim E$ in $\M_n(F)$.  In conclusion, we have $$a\sim \delta\sim A_\delta\sim E$$ in $\mathrm \M_n(K)$; that is,
	 there exists $Q\in \GL_n(K)$ such that $E =Qa Q^{-1}$. By Lemma~\ref{l3.4}, there exist $G\in \GL_n(K)$ and $H\in \mathrm{M}_n(K)$ such that $EGE^{-1}G^{-1}$ and $EH-HE$ are algebraic over $F$ of degree $n$. Put $v=Q^{-1}GQ$ and $u=Q^{-1}H Q$. Then, $$ava^{-1}v^{-1}=a(Q^{-1}GQ)a^{-1}(Q^{-1}GQ)^{-1}$$$$=Q^{-1}(QaQ^{-1}) G (QaQ^{-1})^{-1}G^{-1})Q=Q^{-1} (EGE^{-1}G^{-1})Q,$$ which is algebraic over $F$ of degree $n$. Similarly, $$au-ua=(Q^{-1}EQ)(Q^{-1}HQ)-(Q^{-1}HQ)(Q^{-1}EQ)=Q^{-1}(EH-HE)Q,$$ which is also algebraic of degree $n$ over $F$. The claim is shown.
	 
	 Now, we show that there exist $b,c\in D^*$ such that $aba^{-1}b^{-1}$ and $ac-ca$ are algebraic over $F$ of degree $n$. We first work with the multiplicative commutator $aba^{-1}b^{-1}$. Assume that $axa^{-1}x^{-1}$ is algebraic over $F$ of  bounded degree $s$ for every $x\in D^*$. It is easy to see that $s\le n$. Put $$f(x,y_1,y_2,\dots,y_s)=g_s(axa^{-1}x^{-1}, y_1,y_2,\dots,y_s)$$ where $g_s$ is a generalized  Laurent polynomial we have defined in (1). Then $f$ is non-zero in $\mathrm \M_n(K)_F \left \langle  x,y_1,\dots,y_n\right\rangle $ as it can be seen in \cite[Page 159]{Pa_ChFoLe_04} (or see a more general result in \cite{Pa_HaDuBi_17}). Since $a\alpha a^{-1}\alpha ^{-1}$ is algebraic of degree $\le s$ over $F$ for every $\alpha\in D^*$, by Lemma~\ref{l2.2}, $$f(\alpha,\beta_1,\beta_2,\dots,\beta_s)=g_s(a\alpha a^{-1}\alpha^{-1}, \beta_1,\beta_2,\dots,\beta_s)=0$$ for every $\alpha\in D^*$ and $\beta_i\in D$, which implies that $f=0$ is a generalized Laurent identity of $D$. By Lemma~\ref{l2.3}, $f=0$ is also a generalized Laurent identity of $\mathrm \M_n(K)$; that is, $$g_s(a\alpha a^{-1}\alpha^{-1}, \beta_1,\beta_2,\dots,\beta_s)=f(\alpha,\beta_1,\beta_2,\dots,\beta_s)=0$$ for every $\alpha\in \GL_n(K)^*$ and $ \beta_i\in\mathrm \M_n(K)$. It implies that $a\alpha a^{-1}\alpha^{-1}$ is algebraic of degree $\le s$ for every $\alpha \in \GL_n(K)$. Moreover, $ava^{-1}v^{-1}$ is algebraic over $F$ of degree $n$ by the previous part, so $n\le s$. Thus, $s=n$; that is, there exists $b\in D^*$ such that $aba^{-1}b^{-1}$ is algebraic over $F$ of degree $n$. 
	 
	 Finally, consider the case of  additive commutators. Assume that $ax-xa$ is algebraic over $F$ of  bounded degree $s$ for every $x\in D^*$. Using the same arguments as  above with $$f(x,y_1,y_2,\dots,y_s)=g_s(ax-xa, y_1,y_2,\dots,y_s),$$ we have $s=n$; that is, there exists $c\in D^*$ such that $ac-ca$ is algebraic over $F$ of degree $n$.
	 \end{proof}
	
\section{Algebraicity of bounded degree over the center}	
Using the results of previous sections, here we give the generalizations of some results previously obtained by other authors. As we have mentioned in the introduction, the classical Jacobson theorem \cite[Theorem 7]{Pa_Ja_45} asserted that if a division ring $D$ is algebraic of bounded degree over its center $F$, then $D$ must be centrally finite; that is, $[D:F]<\infty$. In the following theorem, we successfully evaluate this dimension with the weaker assumption on $D$, namely, it is required only the existence of a noncentral normal subgroup of $D^*$ which is algebraic of bounded degree $d$ over $F$. 

\begin{theorem}\label{c4.3} Let $D$ be a division ring with center $F$. Assume that $N$ is a noncentral normal subgroup of $D^*$. If  $N$ is algebraic of bounded degree $d$ over $F$, then $[D:F]\le d^2$.
\end{theorem}
\begin{proof} If $F$ is finite, then for every element $a\in N$, the subfield $F(a)$ is  finite,  which implies that $a$ is an element of finite order in $D^*$. In view of \cite[Theorem 8]{Pa_He_78}, $N$ is central that is a contradiction. Hence, $F$ is infinite. Let $g_d(x,y_1,y_2,\dots,y_d)$ be the generalized Laurent polynomial which is defined in (\ref{e1}). By Lemma~\ref{l2.2},  it follows in particular that $g_d(a,b_1,b_2,\dots,b_d)=0$ for every $a, b_1,b_2,\dots,b_d\in N$. Hence, $g_d=0$ is a generalized Laurent identity of $N$. By \cite{Pa_Ch_96}, $D$ is centrally finite, say,  of dimension $[D:F]=n^2$. By Theorem~\ref{t3.3}, there exist $a\in N$ and $u\in D^*$ such that $F(aua^{-1}u^{-1})$ is a maximal subfield of $D$. According to \cite[Theorem (15.8)]{Bo_La_91}, $[F(aua^{-1}u^{-1}):F]=n$, which implies that $aua^{-1}u^{-1}$ is algebraic of degree $n$ over $F$. Since $aua^{-1}u^{-1}=a(ua^{-1}u^{-1})\in N$ is algebraic of bounded degree $d$ over $F$,  $n\le d$ which implies that $[D:F]\le d^2$. 
\end{proof}

The next result can be considered as a generalization of \cite[Theorem 6]{Pa_ChFoLe_04}.
\begin{theorem}\label{th:4.2} Let $D$ be a division ring with center $F$. Assume that $N$ is a noncentral subnormal subgroup of $D^*$. If all multiplicative commutators of the form $aba^{-1}b^{-1}$, where $a\in N$ and $ b\in D^*$, are algebraic over $F$ of bounded degree $d$,  then $[D:F]\le d^2$.
\end{theorem}
\begin{proof} We claim that $F$ is infinite. Assume by contrary that $F$ is finite. For every $x\in D^*$ and $a\in N$, since $axa^{-1}x$ is algebraic over $F$, the subfield $F(axa^{-1}x^{-1})$ of $D$ is finite. Consequently, $axa^{-1}x^{-1}$ is of finite order in $D^*$. Now, if $ax a^{-1}x^{-1}\in F$ for every $a, x\in N$, then $N$ is soluble, which implies that $N$ is central (e.g., see \cite[Theorem 6 (iii)]{Pa_St_64}), a contradiction. Hence, there exists $\alpha\in N$ such that $a\alpha a^{-1}\alpha^{-1}\not\in F$. Since $F$ is finite and $a\alpha a^{-1}\alpha^{-1}$ is algebraic over $F$,  the subfield $F(a\alpha a^{-1}\alpha^{-1})$ of $D$ is finite. Consequently, $a\alpha a^{-1}\alpha^{-1}$ is of finite order in $D^*$. By \cite[Proposition 2.2]{Pa_BiDu_14}, there exists a centrally finite division subring $D_1$ of $D$ such that $D_1$ contains $a\alpha a^{-1}\alpha^{-1}$ as a noncentral element. Moreover, the fact that $a\alpha a^{-1}\alpha^{-1}\in N$ implies that for every $x\in D_1^*$, the commutator $a\alpha a^{-1}\alpha^{-1} x (a\alpha a^{-1}\alpha^{-1})^{-1}x^{-1}$ is of finite order. In view of  \cite[Corollary 2.10]{Pa_Bi_16}, $a\alpha a^{-1}\alpha^{-1}$ belongs to the center $ Z(D_1)$ of $D_1$, a contradiction. Thus, the claim is shown; that is, $F$ is infinite. Now by \cite[Theorem 11]{Pa_AaAkBi_18}, $[D:F]=n^2$. According to Theorem~\ref{t3.3}, there exist $a\in N$ and $b\in D^*$ such that $aba^{-1}b^{-1}$ is algebraic of degree $n$ over $F$. Hence, $n\le d$, which implies that $[D:F]\le d^2$. 
\end{proof}

The following theorem on additive commutators generalizes \cite[Theorem 2]{Pa_AkArMe_98}. 
\begin{theorem}\label{th:4.3} Let $D$ be a division ring with center $F$. Assume that $N$ is a noncentral subnormal subgroup of $D^*$. If all additive commutators of the form $ac-ca$, where $a\in N$ and $ c\in D$ are algebraic over $F$ of bounded degree $d$, then $[D:F]\le d^2$.
\end{theorem}
\begin{proof}
	We first claim that $D$ is centrally finite. Take an element $a\in N\backslash F$. Then, for every $x\in D^*$, the additive commutator $ax-xa$ is algebraic over $F$ of degree $d$. If $F$ is finite, then the subfield $F(ax-xa)$ of $D$ is finite whose cardinality is less than or equal to $|F|^d$, which implies that  $(ax-xa)^{m!}=1$, where  $m=|F|^d$. Hence, $D$ satisfies a generalized polynomial identity $(ax-xa)^{m!}=1$. It is well known (e.g., see the original one \cite[Theorem 5]{Pa_Ma_69}) that  $D$ is centrally finite. If $F$ is infinite, then $D$ is centrally finite by \cite[Theorem 18]{Pa_AaAkBi_18}. Two these cases lead us to the conclusion that $D$ is centrally finite. Suppose that $[D:F]=n^2$. By Theorem~\ref{t3.3}, there exist $a\in N$ and $c\in D^*$ such that $ac-ca$ is algebraic over $F$ of degree $n$. Thus, $n\le d$, and $[D:F]\le d^2$. 
\end{proof}

\section{Algebraicity of bounded degree over a subfield}

Let $D$ be a division ring with center $F$ and $K$ any subfield of $D$ which is not necessarily central; that is, $K$ is not necessarily contained in $F$. This section is devoted to the study of left algebraicity over $K$. The main result we get in Theorem~ \ref{left algebraic} is very general. In particular, it is an extension  of Theorem~\ref{c4.3}.

\begin{lemma}\label{lem:5.1} Let $D$ be a division ring, $K$ a subfield of $D$ and $x$ an element in $D$ which is left algebraic over $K$.  Then, $x$ is left algebraic of degree $d$ over $K$ if and only if $d$ is the biggest positive integer such that  $\{1,x, \dots,x^{d-1}\} $ is left independent over $K$. 
\end{lemma}
\begin{proof}
			
Assume that $x$ is left algebraic over $K$ of degree $d$. Then, there exist some  elements $a_0, a_1,\dots, a_d\in K$ with $a_d\ne 0$ such that $a_0+a_1x+\dots+a_dx^d=0$. Hence,  $\{1,x, \dots,x^{n}\} $ is left dependent over $K$ for every $n\ge d$. If $\{1,x, \dots,x^{d-1}\}$ is left dependent over $K$, then there exist $a_0, a_1,\dots, a_d\in K$ not all are zeros such that $a_0+a_1x+\dots+a_{d-1}x^{d-1}=0$. This implies that $x$ is left algebraic of degree $<d$. Therefore, $d$ is the biggest positive integer such that $\{1,x, \dots,x^{d-1}\} $ is left independent over $K$.	Conversely, assume that $d$ is the biggest positive integer such that $\{1,x,\dots,x^{d-1}\}$ is left independent over  $K$. Then, $x$ is left algebraic of degree $\ge d$ over $K$. Moreover, since $\{1,x,\dots,x^{d}\}$ is left dependent over  $K$, there exist $a_0, a_1,\dots, a_d\in K$ with $a_d\ne 0$ such that $a_0+a_1x+\dots+a_dx^d=0$, which implies that $x$ is left algebraic of degree $d$ over $K$.  
\end{proof}

\begin{lemma}\label{lem:5.2} Let $K$ be a field and $R$ a ring containing $K$ as a subring. If $R$ is a domain and left algebraic over $K$, then $R$ is a division ring. 
\end{lemma}	
\begin{proof}	
For any $a\in R\backslash \{0\}$, let $f(t)=t^n+a_{n-1}t^{n-1}+\cdots+a_0$ be the left minimal polynomial of $a$ over $K$. Since $f(t)$ is irreducible over $K$, $a_0\ne 0$, and it follows that  
$$a(a^{n-1}+a_{n-1}a^{n-2}+\cdots+a_1)(-a_0^{-1})=1. $$
Thus, we have shown that every non-zero element of $R$ is right invertible, and this implies that every non-zero element of $R$ is invertible. 
\end{proof}

Let $K$ be a field and $R$ a ring containing $K$ as a subring. For an element $a\in R$, if $a$ is left algebraic (resp. right algebraic) over $K$, then $\mathrm{ldeg}_K(a)$ (resp., $\mathrm{rdeg}_K(a)$) is denoted by the degree of the left minimal polynomial (resp., the right minimal polynomial) of $a$ over $K$. In case $K$ is central, this degree is denoted simply by $\deg_K(a)$. Assume that $S$ is a non-empty subset of $R$ which is left algebraic of bounded degree over $K$. Put $\mathrm{ldeg}_K(S)=\max\{\mathrm{ldeg}_K(a)\mid a\in S \}$ and, additionally, if $K$ is central, then put $\mathrm{deg}_K(S)=\max\{\mathrm{deg}_K(a)\mid a\in S \}$. The following lemma plays an important role in this section.
\begin{lemma}\label{l5.2}
	Let $D$ be a centrally finite division ring with center $F$ and $K$ a maximal subfield of $D$. Assume that $N$ is a noncentral normal subgroup of $D^*$. If $[D:F]=n^2$, then $$\deg_F(D)=\deg_F(N)=\mathrm{ldeg}_K(D)=\mathrm{ldeg}_K(N)=n.$$ In particular, if $N$ is left algebraic of bounded degree $d$ over $K$, then $[D:F]\le d^2$.
\end{lemma}
\begin{proof} Assume that $[D:F]=n^2$. For any $a\in D$, let $L$ be a maximal subfield  of $D$ containing $F(a)$. By \cite[Theorem (15.8)]{Bo_La_91}, $\dim_FL=n$, which implies that  $\deg_F(a)\le n$. Hence, $\deg_F(D)\le n$. By Theorem~\ref{t3.3}, there exist $a\in N$ and $b\in D^*$ such that $F(aba^{-1}b^{-1})$ is a maximal subfield of $D$, so $\dim_FF(aba^{-1}b^{-1})=n$. Therefore,  $n\le \deg_F(N)$ because $aba^{-1}b^{-1}\in N$. Thus, $$n\le \deg_F(N)\le \deg_F(D)\le n,$$ which implies that $\deg_F(D)=\deg_F(N)=n$.
	
	It is clear that $\mathrm{ldeg}_K(N)\le \mathrm{ldeg}_K(D)\le \deg_F(D)=n$. To  conclude that $\mathrm{ldeg}_K(N)=\mathrm{ldeg}_K(D)=n$, it suffices to show $n\le \mathrm{ldeg}_K(N)$. Consider $a\in N$ and $b\in D^*$ such that $L=F(\alpha)$ is a maximal subfield of $D$, where $\alpha=aba^{-1}b^{-1}$. Then, $\{1, \alpha, \dots, \alpha^{n-1} \}$ is a basis of $L$ over $F$. Let $f: D\to D$ be a map defined by $f(v)=v\alpha$ for every $v\in D$. Then, $f\in \End(_KD)$. 	Let $K[t]$ be the polynomial ring over $K$. For every $P(t)=k_0+k_1t+\dots+k_mt^m\in K[t]$ and $v\in D$, we define $$P(t).v=k_0\mathrm{Id}(v)+k_1f(v)+\dots+k_mf^m(v).$$ One may show that with this left action, $D$ is a left $K[t]$-module. Observe that if $p_\alpha(t)=a_0+a_1t+\dots+a_{n-1}t^{n-1}+t^{n}$ is the minimal polynomial of $\alpha$ over $F$, then 
	$$p_\alpha(t).v = a_0\mathrm{Id}(v)+a_1f(v)+\dots+a_{n-1}f^{n-1}(v)+f^{n}(v)$$
	
	\hspace*{2.8cm}$=a_0v+a_1v\alpha+\dots+a_{n-1}v\alpha^{n-1}+v\alpha^n$
	
	\hspace*{2.8cm}$=v(a_0+a_1\alpha+\dots+a_{n-1}\alpha^{n-1}+\alpha^n)=0$\\ 
	for every $v\in D$. It implies that $D$ is a periodic left $K[t]$-module. Moreover, as $D$ is finitely generated over $K$, so is $D$ over $K[t]$. Hence, $D$ is periodic and finitely generated left $K[t]$-module. By the fundamental theorem for finitely generated modules over a principal ideal domain, there exists an element  $u\in D^*$ such that if $P(t).u=0$, then $P(t)v=0$ for every $v\in D$. We claim that $\{u, u\alpha,\dots,u\alpha^{n-1} \}$ is left independent over
	$K$.  Assume that there exist $k_0,k_1,\dots, k_{n-1}\in K$ not all are zeros such that $$k_0 u+k_1u\alpha+\dots+ k_{n-1}u\alpha^{n-1}=0.$$ Then, $$(k_0 \mathrm{Id}+k_1f+\dots+ k_{n-1}f^{n-1})(u)=0.$$ By the definition of $u$, one has 
	$$(k_0 \mathrm{Id}+k_1f+\dots+ k_{n-1}f^{n-1})(v)=0$$ for every $v\in D$.	On the other hand, by \cite[Theorem (15.3)]{Bo_La_91}, $D$ is a faithful simple left module over $R=D\otimes _FL$ with the left action of $R$ on $D$ given by $$(d\otimes \ell) \circ v=dv\ell$$ for every $d\otimes \ell\in R$ and $v\in D$. Then, 
	$$ \left(\sum_{i=0}^{n-1} k_i\otimes \alpha^i\right)(v)=\sum_{i=0}^{n-1} k_i v\alpha^{i}=\sum_{i=0}^{n-1} k_i f^i(v)=0$$ for every $v\in D$.
 Since $D$ is faithful, one has $\sum_{i=0}^{n-1} k_i\otimes \alpha^i=0$. It implies that $k_0=k_1=\dots=k_{n-1}=0$, a contradiction, and the claim is shown. Hence, $\{u, u\alpha,\dots,u\alpha^{n-1} \}$ is left independent over $K$, so is $\{1, u\alpha u^{-1},\dots,u\alpha^{n-1}u^{-1} \}$. Using Lemma~\ref{lem:5.1}, we conclude that $u\alpha u^{-1}$ is left algebraic of degree $\ge n$ over $K$. Since $u\alpha u^{-1}\in N$, we have  $n\le \mathrm{ldeg}_K(N)$ as desired. Hence, $$\deg_F(D)=\deg_F(N)=\mathrm{ldeg}_K(D)=\mathrm{ldeg}_K(N)=n.$$
Finally, if $N$ is left algebraic of bounded degree $d$ over $K$ then $$\deg_F(N)=\mathrm{ldeg}_K(N)=n\le d.$$ By Theorem \ref{c4.3}, $[D:F]\le n^2\le d^2$, and the proof is complete.
\end{proof}

Let  $X=\{x_1,x_2,\dots,x_m\}$ be $m$ noncommuting indeterminates and $M$ the free monoid generated by $X$. Every element $w=x_{i_1}\dots x_{i_r}$ of $M$ is called a \textit{word} on $X$ and $r$ is called the \textit{length} of $w$ which is denoted by $\ell(w)$. Recall that $M$ is a totally ordered monoid with the \textit{degree lexicographic order} $x_1>x_2>\cdots>x_m$ which is defined as the following: for $u, v\in M, u<v$ in case $\ell(u)<\ell(v)$; otherwise $u, v$ are compared according to lexicographic order.  It is trivial that $M$ is a multiplicative subsemigroup in the free algebra $F\langle X\rangle $ in $X$ over some field $F$.  Before establishing  the main result of this section, for the convenience of use, we state the following lemma whose proof follows easily from \cite[Theorem 2.4]{Pa_BeDrSh_13}.
	
	\begin{lemma}\label{lemma Bell}
		Let $M$, $m$ and $X$ be as above and assume that $d$ is a positive integer. Then, there exists a natural number $n=n(m,d)$ depending on $m$ and $d$ satisfying the following condition:
		
		If $w$ is a word in $M$ of length $\ell(w)>n$, then 
		\begin{enumerate}
			\item either $w=v_1u^dv_2$,  where $v_1,u,v_2\in M$ and $u$ is non-trivial, or
			\item $w=v_1u_1u_2\dots u_dv_2$, where  $v_1,v_2,u,u_1,u_2,\dots,u_d\in M$ satisfying the condition that  $u_1u_2\dots u_d>u_{\sigma(1)}u_{\sigma(2)}\dots u_{\sigma(d)}$ for any non-trivial permutations $\sigma\in S_d$ of $\{1,2,\dots,d\}$ and $(d-1)\ell (u_i)<\ell(u_1u_2\dots u_d)$ for any $i\in \{1,2,\dots,d\}$. \hspace*{8.25cm} $\square$
		\end{enumerate}
	\end{lemma}
	
	\begin{theorem} \label{left algebraic}
		Let $D$ be a division ring with center $F$, and $K$ a subfield of $D$. Assume that $N$ is a noncentral normal subgroup of $D^*$. If $N$ is left algebraic of bounded degree $d$ over $K$, then $[D:F]\leq d^2$.
	\end{theorem}
	\begin{proof} For a positive integer $m$, let $X=\{x_1,x_2,\dots,x_m\}$ be the set of $m$ noncommuting indeterminates and $M$ the free monoid generated by $X$ with the degree lexicographic order $x_1>x_2>\cdots>x_m$. For a word $w=w(x_1,x_2,\dots,x_m)\in M$ and $a_1,a_2,\dots,a_m\in D^*$, the notation $w(a_1, a_2,\dots, a_m)$ denotes the element of $D^*$ in which every $x_i$ is replaced by $a_i$ in $w$. Let $n=n(m,d)$ be an integer depending on $m$ and $d$ as it was described  in Lemma \ref{lemma Bell}. We claim that for any word $w(x_1,x_2,\dots,x_m)\in M$ and $m$ elements $a_1,a_2,\dots,a_m\in N$, there exist $s$ words $q_1(x_1,x_2,\dots,x_m)$, $q_2(x_1,x_2,\dots,x_m)$, $\dots,q_s(x_1,x_2,\dots,x_m)\in M$ with $\ell(q_i)\leq n$ for $1\leq i\leq s$ such that 
		$$w(a_1,a_2,\dots,a_m)\in\sum_{i=1}^{s}Kq_i(a_1,a_2,\dots,a_m).$$
Assume by contrary that there exists a word in $M$ and $a_1,a_2,\dots,a_m\in N$ which does not satisfy the claim. Put $S=\{ w\in M \mid w \text{ does not sastisfy the claim} \}$.
Let $w=w(x_1,x_2,\dots,x_m)$ be the minimal element of $S$. It is clear that the claim holds for any $v(x_1,x_2,\dots,x_m)\in M$ with $\ell(v)\leq n$, so $\ell(w)>n$. Now, for any positive integer $s$ and $q_1(x_1,x_2,\dots,x_m)$, $q_2(x_1,x_2,\dots,x_m)$, $\dots,$ $q_s(x_1,x_2,\dots,x_m)\in M$ with $\ell(q_i)\leq n$, one has
		\begin{align} \label{e2}
		w(a_1,a_2,\dots,a_m)\notin\sum_{i=1}^{s}Kq_i(a_1,a_2,\dots,a_m).
		\end{align}
According to Lemma \ref{lemma Bell}, there are two cases for $w$ to examine.
\bigskip

\textit{Case 1.}  $w=v_1u^dv_2$. Consider the following elements from $N$: $$p=v_1(a_1,a_2,\dots,a_m), q=u(a_1,a_2,\dots,a_m), r=v_2(a_1,a_2,\dots,a_m).$$ 
Assume that $f(t)=t^k+\alpha_{k-1}t^{k-1}+\dots+\alpha_1t+\alpha_0\in K[t]$ is the left minimal polynomial of the   element $pqp^{-1}$ over $K$. Then, we have 
  \begin{align} \label{e3} 
pq^kp^{-1}+\alpha_{k-1}(pq^{k-1}p^{-1})+\alpha_{k-2}(pq^{k-2}p^{-1})+\cdots+\alpha_1(pqp^{-1})+\alpha_0=0.   \end{align}
Note that $k\le d$ because $N$ is left algebraic of bounded degree $d$ over $K$. Hence, 
by multiplying both sides of (3) with $pq^{d-k}p^{-1}$ in case $k<d$, we would always suppose that  the following relation holds in general 
   \begin{align} \label{e4} 
pq^dp^{-1}+\alpha_{d-1}(pq^{d-1}p^{-1})+\alpha_{d-2}(pq^{d-2}p^{-1})+\cdots+\alpha_1(pqp^{-1})+\alpha_0=0.   \end{align}
Of course, if  $k<d$, then in (4), we must have $\alpha_0=\alpha_1=\dots=\alpha_{d-k-1}=0$. 
In view of (4), we can write
\begin{align*} 
		w(a_1,a_2,\dots,a_m)= &pq^dr=(pq^dp^{-1})pr\\
		= &-\alpha_{d-1}(pq^{d-1}r)-\alpha_{d-2}(pq^{d-2}r)-\cdots-\alpha_1(pqr)-\alpha_0(pr).
\end{align*} 
Hence, if we set $w_i=v_1u^{i-1}v_2$ for $1\le i\le d$, then  	$$w(a_1,a_2,\dots,a_m)\in \sum_{i=1}^{d}Kw_i(a_1,a_2,\dots,a_m).$$ 
\noindent
Since $w$ satisfies (\ref{e2}), there exists at least one index $i\in\{1, 2, \dots, d\}$ such that $w_i$ satisfies (\ref{e2}), which  contradicts the choice of $w$ because $w_i<w$ for every $1\le i\le d$.

\bigskip
\textit{Case 2.}  $w=v_1u_1u_2\dots u_dv_2$, where $u_1u_2\dots u_d>u_{\sigma(1)}u_{\sigma(2)}\dots u_{\sigma(d)}$ for any non-trivial permutations $\sigma\in S_d$,  and $(d-1)\ell(u_i)<\ell(u_1u_2\dots u_d)$ for any $i\in \{1,2,\dots,d\}$. For a non-empty subset $T$ of $\{1,2,\dots,d\}$, put $$p=v_1(a_1,a_2,\dots,a_m), u_T=\sum_{i\in T}u_i, q_T=u_T(a_1,a_2,\dots,a_m), r=v_2(a_1,a_2,\dots,a_m).$$
The condition $(d-1)\ell(u_i)<\ell(u_1u_2\dots u_d)$ for any $i\in \{1,2,\dots,d\}$ implies that  $\ell(u_{i_1}\dots u_{i_k})<\ell(u_1u_2\dots u_d)$ for any $k<d$  and $i_j\in \{1,2,\dots,d\}$. Consequently, the length of any terms of $v_1u_T^kv_2$ with $k<n$ is smaller than $\ell(w)$.	

Observe that $pq_Tp^{-1}$ is left algebraic of bounded degree $d$ over $K$ for every non-empty set $T$ of $\{1,2,\dots,d\}$. Hence, by the same argument as in the proof of Case 1, we would conclude that
\begin{align}\label{e5}
pq_T^dr=\sum_{i=1}^{d}\beta_iw_i(a_1,a_2,\dots,a_m)  
\end{align}
for some  $\beta_i\in K$, and some $w_i\in M$ with $\ell(w_i)<\ell(w), 1\le i\le d$.
The calculation shows that the following formulae holds: 
$$u_1u_2\dots u_d=\sum\limits_{\emptyset \ne T\subseteq \{1,2,\dots,d\} }(-1)^{d-|T|}u_T^d-\sum_{\sigma\in S_d,\sigma\neq Id}u_{\sigma(1)}\dots u_{\sigma(d)}.$$
It implies that
$$w=v_1u_1\dots u_dv_2=\sum_T(-1)^{d-|T|}v_1u_T^dv_2-\sum_{\sigma\in S_d,\sigma\neq Id}v_1u_{\sigma(1)}\dots u_{\sigma(d)}v_2.$$	
Since $w$ satisfies (\ref{e2}), there exists at least one term on the right side satisfying (\ref{e2}). By (\ref{e5}), all terms of the first summation do not satisfy (\ref{e2}), which implies that there exists $Id\ne\sigma\in S_n$ such that the summand $v_1u_{\sigma(1)}\dots u_{\sigma(d)}v_2$ of the second summation satifies (\ref{e2}). But this is impossible because $v_1u_{\sigma(1)}\dots u_{\sigma(d)}v_2<w$ and $w$ is the minimal element of $S$. Thus, the claim is shown.
 
Our second claim is that the subring $F[N]$ of $D$ generated by $N$ over $F$ satisfies the identity $g_d(x,y_1,y_2,\dots,y_d)=0$, where $g_d$ is defined in  (\ref{e1}). Indeed, for any $c,c_1,c_2,\dots,c_d\in F[N]$, there exist elements $a_1,a_2,\dots,a_{m'}\in N$ such that $c,c_1,c_2,\dots,c_d\in F[a_1,a_2,\dots,a_{m'}]$. If $\{a_1,a_2,\dots,a_{m'}\}$ is commuting set; that is, $a_ia_j=a_ja_i$ for every $1\le i,j\le m'$, then $\{c,c_1,c_2,\dots,c_d\}$ is also commuting set, which implies trivially that $g_d(c,c_1,c_2,\dots,c_d)=0$. Now, assume that $\{a_1,a_2,\dots,a_{m'}\}$ is not commuting. Put $L=K[a_1,a_2,\dots,a_{m'}]$ is the subring of $D$ generated by $\{a_1,a_2,\dots,a_m'\}$ over $K$. 		
Then,
$$L\subseteq\sum_{w\in M}Kw(a_1,a_2,\dots,a_{m'}).$$
By the first claim, if $V=\sum_{w\in M}Kw(a_1,a_2,\dots,a_{m'})$, then $\dim(_KV)<\infty$ as a left vector $K$-space. Hence, $\dim(_KL)<\infty$, and in view of  Lemma~\ref{lem:5.2}, $L$ is a division ring. Put $N_1=N\cap L$. Then, $N_1$ is noncentral (because $N$ contains $a_1,a_2,\dots,a_{m'}$) and left algebraic of bounded degree $d$ over $K$. From Lemma~\ref{l5.2}, we conclude that $\dim_{Z(L)}L\leq d^2$, which implies that $L$ satisfies the polynomial identity  $g_d(x,y_1,\dots,y_d)=0$. Hence, $F[a_1,\dots,a_m]$ also satisfies the polynomial identity $g_d(x,y_1,\dots,y_d)=0$. In particular, $g_d(c,c_1,\dots,c_d)=0$.
Thus, the second claim is shown; that is, $g_d(x,y_1,\dots,y_d)=0$ is a polynomial identity of $F[N]$. Using \cite[Theorem 4.6.1]{Bo_BeMaMi_96}, we   conclude that the division subring  $F(N)$ of $D$  generated by $N$ over $F$ satisfies a polynomial identity. As a result, $F(N)$ is finite dimensional over its center (e.g., see  \cite[Theorem 5]{Pa_Ma_69}). Moreover, it is well known that every division ring is generated by any noncentral normal subgroup (see \cite[14.3.8]{Bo_Sc_87} for a general result on noncentral subnormal subgroups); that is, $F(N)=D$. Thus, $D$ is centrally finite. Again, using Lemma~\ref{l5.2}, one has $[D:F]\leq d^2$, as  desired. 
\end{proof}

\textbf{Acknowledgements.}
The authors would like to express their sincere gratitude to the editor and 
the referee for their comments and suggestions.


\begin{thebibliography}{99}
\bibitem{Pa_AaAkBi_18} M. Aaghabali, S. Akbari, M. H. Bien, Division Algebras with Left Algebraic Commutators, \textit{Algebr. Represent. Theor.} \textbf{21} (2018) 807--816.
	
\bibitem{Pa_AaBi_19} M. Aaghabali, M. H. Bien, Certain Simple Maximal Subfields in Division Rings,  \textit{Czechoslovak Math. J.} \textbf{69} (2019), no. 4, 1053--1060.
	
\bibitem{Pa_AkArMe_98} S. Akbari, M. Arian-Nejad and M. L. Mehrabadi, On additive commutator groups in division rings, \textit{Results Math.} \textbf{33}  (1998), 9--21.
		
\bibitem{Pa_BeRo_14} J. P. Bell, D. Rogalski, Free subalgebras of division algebras over uncountable fields, \textit{Math. Z.} \textbf{277} (2014), 591--609.
			
\bibitem{Pa_BeDrSh_13} J. P. Bell, V. Drensky, Y. Sharifi, Shirshov's theorem and division rings that are left algebraic over a subfield, \textit{J. Pure Appl. Algebra} \textbf{217} (2013), 1605--1610.
			
\bibitem{Bo_BeMaMi_96} K.I. Beidar, W.S. Martindale and A.V. Mikhalev, \textit{Rings with Generalized Identities}, Marcel Dekker, Inc., (New York-Basel-Hong Kong, 1996).
			
\bibitem{Pa_Bi_16} M. H. Bien, Subnormal subgroups in division rings with generalized power central group identities, \textit{Arch. Math.} (Basel) \textbf{106} (2016), no. 4, 315--321.
			
\bibitem{Pa_BiDu_14} M. H. Bien, D. H. Dung, On normal subgroups of division rings which are radical over a proper division subring, \textit{Studia Sci. Math. Hungar.} \textbf{51} (2014), no. 2, 231--242.

\bibitem{Bo_BeOg_13} G. Berhuy, F.Oggier, \textit{An Introduction to Central Simple Algebras and Their Applications to Wireless Communication} Mathematical Surveys and Monographs \textbf{191}, (AMS, 2013).
			
\bibitem{Pa_ChFoLe_04} M. A. Chebotar,  Y. Fong and P.-H. Lee, On division rings with algebraic commutators of bounded degree, \textit{Manuscripta Math.} \textbf{113} (2004), 153--164.
			
\bibitem{Pa_Ch_96} K. Chiba, Generalized rational identities of subnormal subgroups of skew fields, \textit{Proc. Amer. Math. Soc.} \textbf{124} (1996), 1649--1653.

\bibitem{Pa_DeBiHa_19} T. T. Deo, M. H. Bien and B. X. Hai, On weakly locally finite division rings, \textit{Acta Math.
Vietnam.} 44 (2019), 553--569.
			
\bibitem{Pa_Di_86} O. Di Vincenzo, A result on derivations with algebraic values, \textit{Canadian Math. Bull.}
\textbf{29} (1986),  432--437. 

\bibitem{Pa_Fa_60} C. Faith, Algebraic division ring extensions, \textit{Proc. Amer. Math. Soc.} \textbf{11} (1960), 43--53.

\bibitem{Pa_GoMa_86} J. Z. Goncalves and A. Mandel, Are there free groups in division rings? \textit{Israel J. Math.} \textbf{53}
(1986), 69--80.

\bibitem{Pa_GoSh_08} J. Z. Goncalves and M. Shirvani, Algebraic elements as free factors in simple Artinian rings,
\textit{Contemp. Math.} \textbf{499} (2008), 121--125.
					
\bibitem{Pa_HaDuBi_17} B. X. Hai, T. H. Dung and M. H. Bien,  Almost subnormal subgroups in division rings with generalized algebraic rational identities, (2017) \textit{arXiv:1709.04774}.

\bibitem{Pa_HaKhBi_2020} B. X. Hai, H. V. Khanh and M. H. Bien Generalized power central group identities in almost subnormal subgroups of $\GL_n(D)$, \textit{Algebra i Analiz} \textbf{31} (2019), No. 4,  225--239, English tranl. in \textit{St. Petersburg Math. J.} \textbf{31} (2020), No. 4, 739--749.
			
\bibitem{Pa_HaTrBi_19} B. X. Hai, V. M. Trang and M. H. Bien, A note on subgroups in a division ring that are left algebraic over a division subring, \textit{Arch. Math.} \textbf{113} (2019), 141--148.
			
\bibitem{Pa_He_82} I. N. Herstein, Derivations of prime rings having power central values, \textit{Contemp. Math.} \textbf{13}  (1982), 163--171.
			
\bibitem{Pa_He_80} I. N. Herstein, Multiplicative commutators in division rings II, \textit{Rend. Circ. Mat. Palermo II} \textbf{29} (1980),  485--489.
			
\bibitem{Pa_He_78} I. N. Herstein,  Multiplicative commutators in division rings, \textit{Israel J. Math.} \textbf{31} (1978), 180--188.
			
\bibitem{Pa_HePrSc_75} I. N. Herstein, C. Procesi and M. Schacher, Algebraic valued functions of noncommutative rings, \textit{J. Algebra} 36, 128--150 (1975).
			
\bibitem{Bo_HoJo_85} R. A. Horn and  C. R. Johnson, \textit{Matrix Analysis}, Cambridge University Press.( Cambridge, UK, 1985).
			
\bibitem{Bo_Ja_64} N. Jacobson, \textit{Structure of rings}, American Mathematical Society, Colloquium Publications, Vol. 37. 190 Hope Street (Providence, 1956). 
			
\bibitem{Pa_Ja_45} N. Jacobson, Structure theory for algebraic algebras of bounded degree, \textit{Ann. of Math.} \textbf{46} (1945), 695--707.
			
\bibitem{Bo_La_91} T.Y.Lam, \textit{A first course in noncommutative rings,} GTM \textbf{131} ( Springer-Vetlag, Berlin, 1991).
					
\bibitem{Pa_Ma_00} M. Mahdavi-Hezavehi, Commutators in division rings revisited, \textit{Iranian Math. Soc.
Bull.} \textbf{26} (2000), 7--88.
					
\bibitem{Pa_MaAkMeHa_95} M. Mahdavi-Hezavehi, S. Akbari-Feyzaabaadi, M. Mehraabaadi and H. Hajie-Abolhassan, On derived groups of Division rings II, \textit{Comm. Algebra} \textbf{23} (1995), 2881--2887.

\bibitem{Pa_MaMa_85} L. Makar-Limanov and P. Malcolmson, Words periodic over the center of a division rings,
\textit{Proc. Amer. Math. Soc.} \textbf{51}, (1985) 590--592.

\bibitem{Pa_Ma_69} W. S. Martindale 3rd,  Prime rings satisfying a generalized polynomial identity, \textit{J. Algebra} \textbf{12} (1969), 576--584.
					
\bibitem{susan_algebraic} S. Montgomery, {Algebraic algebra with involution}, \textit{Proc. Amer. Math. Soc.} \textbf{31} (1972), 368--372.

					
\bibitem{Bo_morandi_96} P. Morandi, \textit{Fields and Galois Theory}, GTM \textbf{167} (Spriger-Verlag, 1996).						
				
\bibitem{Pa_Ro_84} J. D. Rosen,  Generalized Rational Identities and Rings with Involution, \textit{J. Algebra} \textbf{89} (1984), 416--436.

\bibitem{Bo_Rowen_80} L. H. Rowen, \textit{Polynomial identities in ring theory}, Academic Press, Inc. (New York, 1980).
				
\bibitem{Bo_Sc_87} W. R. Scott, \textit{Group theory}, second edition, Dover Publications Inc. (New York,  1987).
					
\bibitem{Pa_Sl_64} M. Slater, On simple rings satisfying a type of “restructed” polynomial identity, {\it J. Algebra} {\bf 1} (1964), 347--354.
						
\bibitem{Pa_St_64} C. J. Stuth, A generalization of the Cartan-Brauer-Hua Theorem, \textit{Proc. Amer. Math. Soc.}
\textbf{15} (2) (1964) 211--217.
	
\end{thebibliography}
\end{document}